\documentclass[12pt,centertags]{amsart}
\usepackage{amsmath,amstext,amsthm,a4,amssymb,amscd}
\usepackage[mathscr]{eucal}
\usepackage{mathrsfs}
\usepackage{epsf}
\numberwithin{equation}{section}

\def\mE{\mathcal{E}}
\def\mF{\mathcal{F}}
\def\mH{\mathcal{H}}

\def\mM{\mathcal{M}}
\def\mN{\mathcal{N}}

\newtheorem{thm}{Theorem}[section]
\newtheorem{lemma}[thm]{Lemma}
\newtheorem{prop}[thm]{Proposition}

\theoremstyle{definition}
\newtheorem{rem}[thm]{Remark}
\theoremstyle{definition}

\theoremstyle{definition}

\newcommand{\be}{\begin{eqnarray}}
\newcommand{\ee}{\end{eqnarray}}

\newcommand{\comment}[1]{}

\begin{document}

\title{Positive scalar curvature on   foliations :  the noncompact case}
 
\author{Guangxiang Su \ and \ Weiping Zhang}

\address{Chern Institute of Mathematics \& LPMC, Nankai
University, Tianjin 300071, P.R. China}
\email{guangxiangsu@nankai.edu.cn,\ weiping@nankai.edu.cn}

\begin{abstract}
Let $(M,g^{TM})$ be a  noncompact (not necessarily complete)   enlargeable Riemannian manifold in the sense of Gromov-Lawson and $F$ an integrable subbundle of $TM$. Let $k^{F}$ be the leafwise scalar curvature associated to $g^F=g^{TM}|_F$. We show that if either $TM$ or $F$ is spin, then ${\rm inf}(k^F)\leq 0$. This generalizes the famous result of Gromov-Lawson on enlargeable manifolds to the case of foliations. It also extends   an ansatz of Gromov on  hyper-Euclidean spaces to general enlargeable Riemannian manifolds, as well as  recent results on compact enlargeable foliated manifolds due to Benameur-Heitsch et al to the noncompact situation. 
\end{abstract}

\maketitle




\section{Introduction} \label{s0}

A classical result of Lichnerowicz \cite{L63} states that if a closed spin manifold carries a Riemannian metric of positive scalar curvature, then its $\widehat A$-genus vanishes. A celebrated result of   Connes \cite[Theorem 0.2]{Co86} generalizes the Lichnerowicz vanishing theorem to the case of spin foliations on an oriented closed manifold. In \cite{Z17}, Zhang gives a differential geometric proof of    the Connes vanishing theorem, and proves an alternative generalization of the Lichnerowicz vanishing theorem to the case of foliations on a closed spin manifold.

In this paper, we deal with the case of {\it noncompact} foliated manifolds. 
Recall that in his comprehensive paper  on  positive scalar curvature,  Gromov \cite[p. 192]{Gr96} has stated  an analogue of the Connes vanishing theorem for any hyper-Euclidean space ${\bf R}^n$, which is noncompact.  
 
\begin{prop}{\rm (}{\bf Gromov Ansatz}{\rm )}
\label{t0.2}   Let $g^{T{\bf R}^n}$ be any   hyper-Euclidean metric on $T{\bf R}^n$ (that is,   for any $X\in T{\bf R}^n$, $|X|_{g^{T{\bf R}^n}}\leq |X|$, where $|X|$ is a standard Euclidean norm of $X$).
Let  
$F\subseteq T{\bf R}^n$ be an integrable subbundle of $T{\bf R}^n$, then the leafwise scalar curvature $k^F$ associated with $g^{T{\bf R}^n}|_F$ verifies  that  ${\rm inf}(k^F)\leq 0$. 
\end{prop}

In view of the fact that ${\bf R}^n$ is the universal covering space of the torus $T^n$, Proposition \ref{t0.2} implies the fact that there is no foliation admitting metric of  positive leafwise scalar curvature  on any torus, which is proved in \cite[Corollary 0.5]{Z17}. It generalizes the famous results of Schoen-Yau \cite{SY79} and Gromov-Lawson \cite{GL80} stating that there is no Riemannian metric of positive scalar curvature on any torus to the case of foliations on torus.

Gromov \cite{Gr96} indicates implicitly that he would combine Connes' proof   of his celebrated vanishing theorem \cite{Co86} with the Gromov-Lawson relative index theorem \cite{GL83} to prove Proposition \ref{t0.2}. However, as we will see below, this might meet potential problems due to the noncompactness of $M$.   
In this paper, we will provide a complete proof of this   Gromov ansatz,   in a more general setting. Our main result generalizes a famous result of Gromov-Lawson to the case of foliations.

In fact, recall that the concept of enlargeability of Gromov-Lawson  has played an important role in their classical papers \cite{GL80}, \cite{GL83} on positive scalar curvature.
Following    \cite[Definition 7.1]{GL83} (see also \cite[Definition 1.1]{CS18}),  a (connected) Riemannian manifold $(M,g^{TM})$   is called enlargeable if for any $\varepsilon>0$, there is a (connected) covering $\widetilde M\rightarrow M$ (with the lifted Riemannian metric of $g^{TM}$) and a smooth   map $f:\widetilde M\rightarrow S^{\dim M}(1)$ of nonzero degree such that $f$ is constant near infinity  and that for any $X\in\Gamma(T\widetilde M)$, one has $|f_*(X)|\leq \varepsilon|X|$. When $M$ is compact, the concept of enlargeability is independent  of   the Riemannian metric $g^{TM}$. However, when $M$ is noncompact, this concept does depend on the chosen metric (cf. \cite[Theorem B]{CS18}).  

A famous result of Gromov-Lawson \cite[Theorem 7.3]{GL83} states that if $(M,g^{TM})$ is  a spin enlargeable complete Riemannian manifold and $k^{TM}$ is the scalar curvature associated to $g^{TM}$, then $\inf(k^{TM})\leq 0$. That is, there is no uniform positive lower bound of $k^{TM}$.

In this paper, we will  generalize  this result to the case of foliations. To be more precise,
let $F \subseteq TM$ be an
integrable subbundle of the tangent bundle $TM$ of $M$.
Let $g^F=g^{TM}|_F$ be the restricted
  Euclidean metric  on $F$, and
$k^F\in C^\infty(M)$  be  the associated leafwise scalar curvature  (cf. \cite[(0.1)]{Z17}).   

With the above notation, the main result of this paper can be stated as follows. 

\begin{thm}\label{t0.1} 
Let $F\subseteq TM$ be an integrable subbundle of the tangent bundle of a spin  enlargeable Riemannian manifold $(M,g^{TM})$ and $k^F$  the leafwise scalar curvature associated to $g^F=g^{TM}|_F$, then one has ${\rm inf}(k^F)\leq 0$.  
\end{thm}

When $F=TM$ and $g^{TM}$ is complete, one recovers the above Gromov-Lawson result. When $M$ is compact, Theorem \ref{t0.1} has been proved (at least in details  for compactly enlargeable manifolds) in \cite{Z19}. So in this paper we will concentrate on the  case of noncompact $M$. Clearly, Proposition \ref{t0.2}   
is a corollary of Theorem \ref{t0.1}. 

It might be interesting to note that we do not assume  $g^{TM}$ to be complete in  Theorem \ref{t0.1}. 

Recall that in the case of noncompact $M$ and $F=TM$, Gromov and Lawson make use of the relative index theorem in \cite[\S 4]{GL83} to prove their result. However, in the case of general $F$ here, even if $(M,F,g^{TM})$ is a  Riemannian foliation, one does not get a uniform positive lower bound of $k^{TM}$ over $M$  if one only assumes that $\inf(k^F)>0$ (although one can always obtain positive $k^{TM}$ on any compact region via an adiabatic process). This indicates that one can not use the relative index theorem  directly to prove Theorem \ref{t0.1}.
To overcome this difficulty, we will use the method developed in \cite{Z19} where a proof without using the relative index theorem   is given for the Gromov-Lawson theorem   on   spin enlargeable manifolds. This amounts to deform the Dirac operator in question by  endomorphisms of the involved (twisted) ${\bf Z}_2$-graded vector bundle, which are invertible near infinity (cf. \cite[(1.11)]{Z19}), thus reducing the noncompact problem to a compact one  by a gluing process.

To be more precise, we  will combine the methods in \cite{Z17} and \cite{Z19}, in an extended form  adapted to the current  noncompact situation,  to prove Theorem \ref{t0.1}.  In particular, the sub-Dirac operators constructed in \cite{LZ01} and \cite{Z17}, as well as   the Connes fibration introduced in \cite{Co86}  (cf. \cite[\S 2.1]{Z17}), will play  essential roles in our proof.  Especially, a key {\it new} deformation of the sub-Dirac operator on the Connes fibration is introduced (cf. (\ref{0.15})). Note also that we do not assume a priori that $g^{TM}$ is complete. This is different with respect to what in \cite{GL83}, where the relative index theorem on complete manifolds \cite[\S 4]{GL83}  plays an essential role.

Moreover, just as in \cite[\S 2.5]{Z17} and \cite{Z19}, our method can also be used to prove the following alternate extension of Gromov-Lawson's result.

\begin{thm}\label{t0.3} 
Let $F\subseteq TM$ be a spin integrable subbundle of the tangent bundle of an     enlargeable Riemannian manifold $(M,g^{TM})$ and $k^F$  the leafwise scalar curvature associated to $g^F=g^{TM}|_F$, then one has ${\rm inf}(k^F)\leq 0$.  
\end{thm}

When $M$ is compact and the homotopy groupoid of $(M,F)$ is Hausdorff, Theorem \ref{t0.3} is due to  Benameur-Heitsch  \cite{BH19}. In \cite{Z19}, Zhang eliminates this  Hausdorff condition.  Thus Theorem \ref{t0.3} is new in  the case of noncompact $M$.

\section{Proof of Theorems \ref{t0.1} and \ref{t0.3}}\label{s1}

 In this section, we prove Theorems \ref{t0.1} and \ref{t0.3}. In Section \ref{s1.1}, we recall the basic geometric set up. In Section \ref{s1.2}, we lift things to the Connes fibration. In Section \ref{s1.3}, we study the deformed sub-Dirac operators on the Connes fibration. In Section \ref{s1.4}, we  prove Theorem \ref{t0.1} in 
detail.    The proof of Theorem \ref{t0.3} is similar.

\subsection{Foliations on noncompact enlargeable manifolds
}\label{s1.1} 

Let $M$ be a noncompact (connected) smooth manifold and $g^{TM}$ a Euclidean metric on the tangent bundle $TM$. Following \cite[Definition 7.1]{GL83} (see also \cite[Definition 1.1]{CS18}), we say that the Riemannian metric $g^{TM}$ is   enlargeable if for any $\varepsilon>0$, there is a (connected) covering $\pi_\varepsilon:  M_\varepsilon\rightarrow M$ and a smooth map $f_\varepsilon: M_\varepsilon \rightarrow S^{\dim M}(1)$, where $S^{\dim M}(1)$ is the standard unit sphere of dimension $\dim M$,  such that there exists a compact subset $K_\varepsilon\subset M_\varepsilon$ verifying that $f_\varepsilon$ is constant on $M_\varepsilon\setminus K_\varepsilon$  and   $\deg(f_\varepsilon)\neq 0$.  Moreover,   for any $X\in \Gamma(T  M_\varepsilon)$, $f$ should satisfy
\begin{align}\label{0.1}
\left|f_{\varepsilon*}(X)\right|\leq \varepsilon|X|_{g^{T M_\varepsilon}},
\end{align}
where $g^{T M_\varepsilon}$ is the lifted metric   $\pi_\varepsilon^*g^{TM}$.

From now on, we assume that $g^{TM}$ is   enlargeable. Without loss of generality, we assume that for any $\varepsilon>0$, $f_\varepsilon(M_\varepsilon\setminus K_\varepsilon)=x_0\in S^{\dim M}(1)$, where $x_0$ is a fixed point on $S^{\dim M}(1)$. 


Without loss of generality, we assume that $\dim M$ is even.\footnote{One may consider $M\times S^1$ if $\dim M$ is odd.}

Let $F$ be an integrable subbundle of the tangent bundle $TM$. Let $g^F=g^{TM}|_F$ be the induced Euclidean metric on $F$. Let $k^F\in C^\infty(M)$ be the   leafwise scalar curvature associated to $g^F$ (cf. \cite[(0.1)]{Z17}).  

We will prove Theorems \ref{t0.1} and \ref{t0.3}  by contradiction. Thus we  assume first that 
there is  $\delta>0$ such that 
\begin{align}\label{0.2}
k^F\geq \delta\ \ \ {\rm over}\ \ M.
\end{align}

Let $F^\perp$ be the orthogonal complement to $F$, i.e., we have the orthogonal splitting 
 \begin{align}\label{0.3}
TM=F\oplus F^\perp,\ \ \ g^{TM}=g^F\oplus g^{F^\perp}.
\end{align}

Let $(E_0,g^{E_0} )$ be a Hermitian vector bundle on $S^{\dim M}(1)$ verifying
\begin{align}\label{0.4}
\left\langle {\rm ch}\left(E_0\right),\left[S^{\dim M}(1)\right]\right\rangle\neq 0 
\end{align}
  and carrying a Hermitian connection $\nabla^{E_0}$.   Let $(E_1={\bf C}^k|_{S^{\dim M}(1)},g^{E_1},\nabla^{E_1})$, with $k={\rm rk}(E_0)$, be the canonical Hermitian trivial vector bundle, with trivial Hermitian connection, on $S^{\dim M}(1)$.  Let $w\in \Gamma({\rm Hom}(E_0, E_1))$ be an endomorphism such that $w|_{x_0}:E_0|_{x_0}\rightarrow E_1|_{x_0}$ is an isomorphism. Let $w^*: \Gamma(E_1)\rightarrow \Gamma(E_0)$ be the adjoint of $w$ with respect to   $g^{E_0}$ and $g^{E_1}$. Set 
\begin{align}\label{0.5}
W=w+w^*. 
\end{align}
Then the self-adjoint endomorphism $W:\Gamma(E_0\oplus E_1)\rightarrow \Gamma(E_0\oplus E_1)$ is invertible near $x_0$.

Let $\varepsilon>0$ be fixed temporarily. 

Let $(M_\varepsilon, F_\varepsilon)=\pi_\varepsilon^*(M,F)$ be the lifted foliation, with $g^{F_\varepsilon}=\pi^*_\varepsilon g^F$ being the lifted Euclidean metric on $F_\varepsilon$. The splitting (\ref{0.3}) lifts canonically to an orthogonal  splitting
 \begin{align}\label{0.5a}
TM_\varepsilon=F_\varepsilon\oplus F_\varepsilon^\perp,\ \ \ g^{TM_\varepsilon}=g^{F_\varepsilon}\oplus g^{F_\varepsilon^\perp}.
\end{align}

 Following \cite{GL83}, we take a compact hypersurface $H_\varepsilon\subset M_\varepsilon\setminus K_\varepsilon$, cutting $M_\varepsilon$ into two parts such that   the compact part, denoted by $M_{H_\varepsilon}$, contains $K_\varepsilon$. Then   $M_{H_\varepsilon}$ is a compact smooth manifold with boundary $H_\varepsilon$.

Let $M'_{H_\varepsilon}$ be another copy of $M_{H_\varepsilon}$. We glue $M_{H_\varepsilon}$ and $M'_{H_\varepsilon}$ along $H_\varepsilon$ to get the double, which we denote by $\widehat M_{H_\varepsilon}$.

\subsection{The Connes fibration}\label{s1.2}

Following  \cite[\S 5]{Co86} (cf. \cite[\S 2.1]{Z17}), let $\widetilde \pi_\varepsilon:\mM_\varepsilon\rightarrow M_\varepsilon$ be the Connes
fibration over $M_\varepsilon$ such that for any $x\in M_\varepsilon$, $\mM_{\varepsilon,x}=\widetilde\pi_\varepsilon^{-1}(x)$
is the space of Euclidean metrics on the linear space $T_xM_\varepsilon/F_{\varepsilon,x}$.
Let  $T^V\mM_\varepsilon$ denote the vertical tangent bundle of the fibration
$\widetilde\pi_\varepsilon:\mM_\varepsilon\rightarrow M_\varepsilon$. Then it carries a natural metric
$g^{T^V\mM_\varepsilon}$ such that   any two points $p,\,q\in \mM_{\varepsilon,x}$ with $x\in M_\varepsilon$ can be joined by a unique geodesic along $\mM_{\varepsilon,x}$. Let $d^{\mM_{\varepsilon,x}}(p,q)$ denote the length of this geodesic.  
  
By  using the Bott connection
  on $TM_\varepsilon/F_\varepsilon$ (cf. \cite[(1.2)]{Z17}), which is leafwise flat, one  lifts $F_\varepsilon$ to an integrable subbundle
$\mF_\varepsilon$ of $T\mM_\varepsilon$. 
  Then $g^{F_\varepsilon}$   lifts to a Euclidean metric $g^{\mF_\varepsilon}=\widetilde \pi^*_\varepsilon g^{F_\varepsilon} $ on $\mF_\varepsilon$.

Let $\mF_{\varepsilon,1}^\perp\subseteq T\mM_\varepsilon$ be a subbundle, which is  transversal to $\mF_\varepsilon\oplus T^V\mM_\varepsilon$,   such that we have a
splitting $T\mM_\varepsilon=(\mF_\varepsilon \oplus T^V \mM_\varepsilon)\oplus\mF_{\varepsilon,1}^\perp$. Then
$\mF_{\varepsilon,1}^\perp$ can be identified with $T\mM_\varepsilon/(\mF_\varepsilon \oplus T^V \mM_\varepsilon)$
and carries a canonically induced metric $g^{\mF_{\varepsilon,1}^\perp}$.
We denote  $\mF_{\varepsilon,2}^\perp=T^V\mM_\varepsilon$.\footnote{We may well assume that $T\mM_\varepsilon=\mF_\varepsilon\oplus\mF_{\varepsilon,1}^\perp\oplus\mF_{\varepsilon,2}^\perp$ is lifted from $T\mM=\mF\oplus\mF_1^\perp\oplus\mF_2^\perp$ via $\pi^*_\varepsilon$, where $\mM$ is the Connes fibration over $M$ as in \cite[\S 2.1]{Z17}.}

  The metric $g^{F_\varepsilon^\perp}$   in (\ref{0.5a}) determines a canonical embedded section $s: M_\varepsilon\hookrightarrow \mM_\varepsilon$.
For any $p\in\mM_\varepsilon$, set $\rho(p)=d^{\mM_{\varepsilon,\widetilde\pi_\varepsilon(p)}}(p,s(\widetilde\pi_\varepsilon(p) ))$.

For any $  \beta,\ \gamma>0$,  following  \cite[(2.15)]{Z17}, let $g_{\beta,\gamma}^{T\mM_\varepsilon}$ be the   metric    on $T\mM_\varepsilon$  defined by
the
  orthogonal splitting,
\begin{align}\label{0.6}\begin{split}
       T\mM_\varepsilon =   \mF_\varepsilon\oplus \mF^\perp_{\varepsilon,1}\oplus \mF^\perp_{\varepsilon,2},  \ 
\  \  \
g^{T\mM_\varepsilon}_{\beta,\gamma}= \beta^2   g^{\mF_\varepsilon}\oplus\frac{
g^{\mF^\perp_{\varepsilon,1}}}{ \gamma^2 }\oplus g^{\mF^\perp_{\varepsilon,2}}.\end{split}
\end{align}

For any $R>0$, let $ \mM_{\varepsilon,R}$ be the smooth manifold with boundary
defined by
\begin{align}\label{0.7} 
\mM_{\varepsilon,R}=\left\{ p\in \mM_\varepsilon\ :\  \rho(p)\leq R \right\}.
\end{align}

Set $\mH_\varepsilon = (\widetilde\pi_\varepsilon)^{-1}(H_\varepsilon)$ and 
\begin{align}\label{0.8} 
\mM_{\mH_\varepsilon,R} = \left(\left(\widetilde\pi_\varepsilon\right)^{-1}
\left(M_{H_\varepsilon}\right)\right)\cap \mM_{\varepsilon,R}.
\end{align}

Consider another copy $\mM_{\mH_\varepsilon,R} '$ of $\mM_{\mH_\varepsilon,R} $. We glue $\mM_{\mH_\varepsilon,R} $ and $\mM_{\mH_\varepsilon,R} '$ along 
$\mH_\varepsilon \cap \mM_{\varepsilon,R}$ to get the double, denoted by $\widehat\mM_{\mH_\varepsilon, R}$, which is a smooth manifold with boundary. Moreover, $\widehat\mM_{\mH_\varepsilon, R}$ is a disk bundle over $\widehat M_{H_\varepsilon}$.  Without loss of generality, we assume that $\widehat\mM_{\mH_\varepsilon, R}$ is oriented. 
Let $g_{\beta,\gamma}^{T\widehat \mM_{\mH_\varepsilon,R}}$ be a metric on $T\widehat \mM_{\mH_\varepsilon,R}$ such that $g_{\beta,\gamma}^{T\widehat \mM_{\mH_\varepsilon,R}}|_{\mM_{\mH_\varepsilon,R}}=g_{\beta,\gamma}^{T\mM_\varepsilon}|_{\mM_{\mH_\varepsilon,R}}$. The existence of $g_{\beta,\gamma}^{T\widehat \mM_{\mH_\varepsilon,R}}$ is clear.\footnote{Here we need not assume that  $g_{\beta,\gamma}^{T\widehat \mM_{\mH_\varepsilon,R}}$ is of product structure near ${\mH_\varepsilon}$.} 

Let $\partial \widehat \mM_{\mH_\varepsilon,R}$ bound another oriented manifold $  \mN_{\varepsilon,R}$ so that $\widetilde  \mN_{\varepsilon,R}=\widehat \mM_{\mH_\varepsilon,R}\cup \mN_{\varepsilon,R}$ is an oriented closed manifold. 
Let $g_{\beta,\gamma}^{T\widetilde \mN_{\varepsilon,R}}$ be a smooth metric on $T\widetilde \mN_{\varepsilon,R}$ so that
$g_{\beta,\gamma}^{T\widetilde \mN_{\varepsilon,R}}
|_{\widehat\mM_{\mH_\varepsilon,R}}=g_{\beta,\gamma}^{T\widehat\mM_{\mH_\varepsilon,R} }$. The existence of $g_{\beta,\gamma}^{T\widetilde \mN_{\varepsilon,R}}$ is clear.

We extend $f_\varepsilon:M_{H_\varepsilon}\rightarrow S^{\dim M}(1)$ to $f_\varepsilon:\widehat M_{H_\varepsilon}\rightarrow S^{\dim M}(1)$ by setting $f_\varepsilon(M_{H_\varepsilon}')=x_0$. Let $\widehat f_{\varepsilon,R}:\widehat \mM_{\mH_\varepsilon,R}\rightarrow S^{\dim M}(1)$ be the smooth map defined by
\begin{align}\label{0.9} 
\widehat f_{\varepsilon,R} = f_\varepsilon\circ \widetilde \pi_\varepsilon \ \ {\rm on} \ \ \mM_{\mH_\varepsilon,R}
\end{align}
and $\widehat f_{\varepsilon,R} (\mM_{\mH_\varepsilon,R}')=x_0$. 

For $i=0,\,1$, let $(\mE_{\varepsilon,R,i},g^{\mE_{\varepsilon,R,i}},\nabla^{\mE_{\varepsilon,R,i}})= \widehat f_{\varepsilon,R}^* (E_i,g^{E_i},\nabla^{E_i})$ be the induced Hermitian vector bundle with Hermitian connection on $\widehat \mM_{\mH_\varepsilon,R}$. Then $\mE_{\varepsilon,R}=\mE_{\varepsilon,R,0}\oplus \mE_{\varepsilon,R,1}$ is a ${\bf Z}_2$-graded Hermitian vector bundle over $\widehat \mM_{\mH_\varepsilon,R}$.

\subsection{Adiabatic limits and deformed sub-Dirac operators on
$\widehat \mM_{\mH_\varepsilon,R}$}\label{s1.3}

We assume first that $TM$ is oriented and spin. Then $TM_\varepsilon=\pi_\varepsilon^*(TM)$ is spin, and thus $\mF_\varepsilon\oplus \mF_{\varepsilon,1}^\perp=\widetilde\pi^*_\varepsilon(TM_\varepsilon)$ is spin.  Without loss of generality, we assume   $F_\varepsilon$ is oriented. Then   $F^\perp_\varepsilon$  is also oriented.  Without loss of generality, we assume that      $\dim \mM_\varepsilon$ is even. 

It is clear that $\mF_\varepsilon\oplus \mF^\perp_{\varepsilon,1},\, \mF^\perp_{\varepsilon,2}$  over $\mM_{\mH_\varepsilon,R}$ can be extended to $\mM_{\mH_\varepsilon,R}'$ such that we have the orthogonal splitting
\begin{align}\label{0.10}
T\widehat \mM_{\mH_\varepsilon,R}= \left(\mF_\varepsilon\oplus\mF^\perp_{\varepsilon,1} \right)\oplus \mF^\perp_{\varepsilon,2}\ \  {\rm on}\ \ \widehat \mM_{\mH_\varepsilon,R}.
\end{align}

Let $S_{\beta,\gamma}(\mF_\varepsilon\oplus \mF^\perp_{\varepsilon,1})$ denote the spinor bundle over $\widehat \mM_{\mH_\varepsilon,R}$ with respect to the metric  $g^{T\widehat\mM_{\mH_\varepsilon,R}}_{\beta,\gamma}|_{\mF_\varepsilon\oplus\mF^\perp_{\varepsilon,1}} $ (thus with respect to
$\beta^2g^{\mF_\varepsilon}\oplus \frac{g^{\mF^\perp_{\varepsilon,1}}}{\gamma^2}$ on $\mM_{\mH_\varepsilon,R}$). Let $\Lambda^* (\mF_{\varepsilon,2}^\perp )$ denote the exterior algebra bundle of $\mF_{\varepsilon,2}^{\perp,*}$, with the   ${\bf Z}_2$-grading given by the natural parity (cf. \cite[(1.15)]{Z19}).

Let $D _{\mF_\varepsilon\oplus\mF_{\varepsilon,1}^\perp,\beta,\gamma}:\Gamma (S_{\beta,\gamma} (\mF_\varepsilon\oplus\mF_{\varepsilon,1}^\perp)\widehat\otimes
\Lambda^* (\mF_{\varepsilon,2}^\perp ) )
\rightarrow
\Gamma (S_{\beta,\gamma} (\mF_\varepsilon\oplus\mF_{\varepsilon,1}^\perp)\widehat\otimes
\Lambda^* (\mF_{\varepsilon,2}^\perp ) )$ be the sub-Dirac operator on $\widehat \mM_{\mH_\varepsilon,R}$ constructed as in \cite[(2.16)]{Z17}. Then it is clear that one can define canonically the twisted sub-Dirac operator (twisted by $\mE_{\varepsilon,R}$) on $\widehat \mM_{\mH_\varepsilon,R}$,
\begin{multline}\label{0.11}
D ^{\mE_{\varepsilon,R}}_{\mF_\varepsilon\oplus\mF_{\varepsilon,1}^\perp,\beta,\gamma}:\Gamma  \left(S_{\beta,\gamma}  \left(\mF_\varepsilon\oplus\mF_{\varepsilon,1}^\perp\right)\widehat\otimes
\Lambda^*\left (\mF_{\varepsilon,2}^\perp \right)\widehat \otimes \mE_{\varepsilon,R}\right )
\\
\longrightarrow
\Gamma  \left(S_{\beta,\gamma}  \left(\mF_\varepsilon\oplus\mF_{\varepsilon,1}^\perp\right)\widehat\otimes
\Lambda^*\left (\mF_{\varepsilon,2}^\perp \right)\widehat \otimes \mE_{\varepsilon,R}\right ).
\end{multline}
Moreover, 
by  \cite[(2.28)]{Z17}, one sees that   the following identity holds on $\mM_{\mH_\varepsilon,R}$, using the same notation for Clifford actions as in \cite{Z17},  
\begin{multline}\label{0.12} 
 \left(D ^{\mE_{\varepsilon,R}}_{\mF_\varepsilon\oplus\mF_{\varepsilon,1}^\perp,\beta,\gamma}\right)^2=-\Delta^{\mE,\varepsilon,\beta,\gamma}+\frac{k^{\mF_\varepsilon}}{4\beta^2}+\frac{1}{2\beta^2}\sum_{i,\,j=1}^{{\rm rk}(F)}R^{\mE_{\varepsilon,R}}(f_i,f_j)c_{\beta,\gamma}\left(\beta^{-1}f_i\right)c_{\beta,\gamma}\left(\beta^{-1}f_j\right)
\\
+
O_{\varepsilon,R}\left(\frac{1}{\beta}+\frac{\gamma^2}{\beta^2}\right)
,
\end{multline}
where $-\Delta^{\mE,\varepsilon,\beta,\gamma}\geq 0$ is the corresponding Bochner Laplacian,  
\begin{align}\label{0.12a}
 k^{\mF_\varepsilon}=\widetilde \pi_\varepsilon^*\left(\pi^*_\varepsilon \left(k^F\right)\right)\geq \delta,
\end{align}
 $R^{\mE_{\varepsilon,R}}
=(\nabla^{\mE_{\varepsilon,R,0}})^2+(\nabla^{\mE_{\varepsilon,R,1}})^2$ and $f_1,\,\cdots,\,f_{{\rm rk}(F)}$ is an orthonormal basis of $(\mF_\varepsilon,g^{\mF_\varepsilon})$.  The subscripts in $O_{\varepsilon,R}(\cdot)$ mean that the estimating constant may depend on $\varepsilon$ and $R$. 

On the other hand, since $g^{\mF_\varepsilon}=\widetilde \pi_\varepsilon^*g^{F_\varepsilon}$, one has via (\ref{0.1}) and (\ref{0.9}) that
\begin{align}\label{0.12a}
 R^{\mE_{\varepsilon,R}}(f_i,f_j)=\sum_{k=0}^1\widehat f_{\varepsilon,R}^*\left( \left(\nabla^{E_k}\right)^2
\left(\widehat f_{\varepsilon,R\,*}( f_i), \widehat f_{\varepsilon,R\,*} (f_j)\right)\right)
 =O\left(\varepsilon^2\right),
\end{align}
where the estimating constant does not depend on $\varepsilon$ and $R$. 

Let $f:[0,1]\rightarrow [0,1]$ be a smooth function such that  $f(t)= 0$ for $0\leq t\leq \frac{1}{4}$, while $f(t) =1$ for $   \frac{1}{2}\leq t\leq 1$.  Let $h:[0,1]\rightarrow [0,1]$ be a smooth function such that $h(t)=1$ for $0\leq t\leq \frac{3}{4}$, while $h(t)=0$ for $\frac{7}{8}\leq t\leq 1$. 

For any $p\in \mM_{\mH_\varepsilon,R}$, we connect $p$ and $s(\widetilde\pi_\varepsilon(p))$ by the unique geodesic in $\mM_{\varepsilon, \widetilde\pi_\varepsilon(p)}$. Let $\sigma(p)\in \mF_{\varepsilon,2}^\perp|_p$ denote the unit vector tangent to this geodesic. Then  
\begin{align}\label{0.13}
 \widetilde \sigma = f\left(\frac{\rho}{R}\right)\sigma
\end{align}
is a smooth section of $\mF_{\varepsilon,2}^\perp|_{\mM_{\mH_\varepsilon,R}}$. It extends to a smooth section of $\mF_{\varepsilon,2}^\perp|_{\widehat\mM_{\mH_\varepsilon,R}}$, which we still denote by $\widetilde\sigma$.  It is easy to see that we may and we will assume that $\widetilde\sigma$ is transversal to (and thus no where zero on) $\partial \widehat\mM_{\mH_\varepsilon,R}$.

The Clifford action $\widehat c(\widetilde\sigma)$ (cf. \cite[(1.47)]{Z17}) now acts on $S_{\beta,\gamma}   (\mF_\varepsilon\oplus\mF_{\varepsilon,1}^\perp )\widehat\otimes
\Lambda^* (\mF_{\varepsilon,2}^\perp  )\widehat \otimes \mE_{\varepsilon,R} $ over $\widehat\mM_{\mH_\varepsilon,R}$. 

We also set
\begin{align}\label{0.14}
 W_{\widehat f_{\varepsilon,R}}=\widehat f_{\varepsilon,R} ^*(W),
\end{align}
where $W$ is defined in (\ref{0.5}). Then $W_{\widehat f_{\varepsilon,R}}$ is an odd endomorphism of $\mE_{\varepsilon,R}$ and thus also acts on $S_{\beta,\gamma}   (\mF_\varepsilon\oplus\mF_{\varepsilon,1}^\perp )\widehat\otimes
\Lambda^* (\mF_{\varepsilon,2}^\perp  )\widehat \otimes \mE_{\varepsilon,R} $ in an obvious way.

Inspired by \cite[(2.21)]{Z17} and \cite[(1.11)]{Z19}, we introduce the following deformation of  $D ^{\mE_{\varepsilon,R}}_{\mF_\varepsilon\oplus\mF_{\varepsilon,1}^\perp,\beta,\gamma}$ on 
$\widehat\mM_{\mH_\varepsilon,R}$\ :
\begin{align}\label{0.15}
 D ^{\mE_{\varepsilon,R}}_{\mF_\varepsilon\oplus\mF_{\varepsilon,1}^\perp,\beta,\gamma}
+\frac{\widehat c(\widetilde\sigma)}{\beta}
+\frac{W_{\widehat f_{\varepsilon,R}}}{\beta}.
\end{align}

For this deformed sub-Dirac operator, we have the following analogue of
\cite[Lemma 2.4]{Z17}.

\begin{lemma}\label{t0.4} There exist $c_0>0$, $\varepsilon>0$ and $R>0$ such that   when $\beta,\,\gamma>0$ (which may depend on $\varepsilon$ and $R$) are small enough, 

(i) for any $s\in \Gamma  (S_{\beta,\gamma}   (\mF_\varepsilon\oplus\mF_{\varepsilon,1}^\perp )\widehat\otimes
\Lambda^* (\mF_{\varepsilon,2}^\perp )\widehat \otimes \mE_{\varepsilon,R}  )$ supported in the interior of $\widehat\mM_{\mH_\varepsilon,R}$, one has\footnote{The norms below   depend on $\beta$ and $\gamma$. In case of no confusion, we omit the subscripts for simplicity.}
\begin{align}\label{0.16}
 \left\|\left(D ^{\mE_{\varepsilon,R}}_{\mF_\varepsilon\oplus\mF_{\varepsilon,1}^\perp,\beta,\gamma}
+\frac{\widehat c(\widetilde\sigma)}{\beta}
+\frac{W_{\widehat f_{\varepsilon,R}}}{\beta}\right)s\right\|\geq \frac{c_0 }{\beta}\|s\|;
\end{align}

(ii) for any $s\in \Gamma  (S_{\beta,\gamma}   (\mF_\varepsilon\oplus\mF_{\varepsilon,1}^\perp )\widehat\otimes
\Lambda^* (\mF_{\varepsilon,2}^\perp )\widehat \otimes \mE_{\varepsilon,R}  )$
supported in the interior of $ \mM_{\mH_\varepsilon,R}\setminus  \mM_{\mH_\varepsilon,\frac{R}{2}}$, one has
\begin{align}\label{0.17}
 \left\|\left(h\left(\frac{\rho}{R}\right) D ^{\mE_{\varepsilon,R}}_{\mF_\varepsilon\oplus\mF_{\varepsilon,1}^\perp,\beta,\gamma}
 h\left(\frac{\rho}{R}\right) 
+\frac{\widehat c\left(\widetilde\sigma\right)}{\beta}
+\frac{W_{\widehat f_{\varepsilon,R}}}{\beta}\right)s\right\|\geq \frac{c_0 }{\beta}\|s\|.
\end{align}
\end{lemma}

\begin{proof} Recall that $x_0\in S^{\dim M}(1)$ is fixed and $W|_{x_0}$ is invertible. Let $U_{x_0}\subset S^{\dim M}(1)$ be a (fixed) sufficiently small open neighborhood  of $x_0$ such that the following inequality holds on $U_{x_0}$, where $\delta_1>0$ is a fixed constant,
\begin{align}\label{0.18}
 W^2\geq \delta_1.
\end{align}

Following \cite{Z19}, let $\psi:S^{\dim M}(1)\rightarrow [0,1]$ be a smooth function such that $\psi=1$ near $x_0$ and ${\rm Supp}(\psi)\subset U_{x_0}$. Then 
\begin{align}\label{0.19}
 \varphi_{\varepsilon,R} =1 - \widehat f_{\varepsilon,R}^*\psi 
\end{align}
is a smooth function on $\widehat\mM_{\mH_\varepsilon,R}$ such that $\varphi_{\varepsilon,R}=0 $ on $\mM_{\mH_\varepsilon,R}'$.

Following   \cite[p. 115]{BL91}, let $\varphi_{\varepsilon,R,1},\, \varphi_{\varepsilon,R,2}: \widehat\mM_{\mH_\varepsilon,R}\rightarrow [0,1]$ be defined by
 \begin{align}\label{0.20}
 \varphi_{\varepsilon,R,1} =\frac{\varphi_{\varepsilon,R}}{\left(\varphi_{\varepsilon,R}^2+\left(1-\varphi_{\varepsilon,R}\right)^2\right)^{\frac{1}{2}}},\ \ \ \varphi_{\varepsilon,R,2} =\frac{1-\varphi_{\varepsilon,R}}{\left(\varphi_{\varepsilon,R}^2+\left(1-\varphi_{\varepsilon,R}\right)^2\right)^{\frac{1}{2}}}.
\end{align}
Then $ \varphi_{\varepsilon,R,1}^2+ \varphi_{\varepsilon,R,2}^2=1$. Thus, for any
 $s\in \Gamma  (S_{\beta,\gamma}   (\mF_\varepsilon\oplus\mF_{\varepsilon,1}^\perp )\widehat\otimes
\Lambda^* (\mF_{\varepsilon,2}^\perp )\widehat \otimes \mE_{\varepsilon,R}  )$ supported in the interior of $\widehat\mM_{\mH_\varepsilon,R}$, one has
\begin{multline}\label{0.21}
 \left\|\left(D ^{\mE_{\varepsilon,R}}_{\mF_\varepsilon\oplus\mF_{\varepsilon,1}^\perp,\beta,\gamma}+\frac{\widehat c(\widetilde\sigma)}{\beta}
+\frac{W_{\widehat f_{\varepsilon,R}}}{\beta}\right)s\right\|^2 \\
=\left\|\varphi_{\varepsilon,R,1}\left(D ^{\mE_{\varepsilon,R}}_{\mF_\varepsilon\oplus\mF_{\varepsilon,1}^\perp,\beta,\gamma}+\frac{\widehat c(\widetilde\sigma)}{\beta}
+\frac{W_{\widehat f_{\varepsilon,R}}}{\beta}\right)s\right\|^2 
+\left\| \varphi_{\varepsilon,R,2}\left(D ^{\mE_{\varepsilon,R}}_{\mF_\varepsilon\oplus\mF_{\varepsilon,1}^\perp,\beta,\gamma}+\frac{\widehat c(\widetilde\sigma)}{\beta}
+\frac{W_{\widehat f_{\varepsilon,R}}}{\beta}\right)s\right\|^2 ,
\end{multline}
from which one gets, 
\begin{multline}\label{0.22}
 \sqrt{2}\left\|\left(D ^{\mE_{\varepsilon,R}}_{\mF_\varepsilon\oplus\mF_{\varepsilon,1}^\perp,\beta,\gamma}+\frac{\widehat c(\widetilde\sigma)}{\beta}
+\frac{W_{\widehat f_{\varepsilon,R}}}{\beta}\right)s\right\|
\\
\geq
\left\|\varphi_{\varepsilon,R,1}\left(D ^{\mE_{\varepsilon,R}}_{\mF_\varepsilon\oplus\mF_{\varepsilon,1}^\perp,\beta,\gamma}+\frac{\widehat c(\widetilde\sigma)}{\beta}
+\frac{W_{\widehat f_{\varepsilon,R}}}{\beta}\right)s\right\|
+\left\| \varphi_{\varepsilon,R,2}\left(D ^{\mE_{\varepsilon,R}}_{\mF_\varepsilon\oplus\mF_{\varepsilon,1}^\perp,\beta,\gamma}+\frac{\widehat c(\widetilde\sigma)}{\beta}
+\frac{W_{\widehat f_{\varepsilon,R}}}{\beta}\right)s\right\|
\\
\geq \left\|\left(D ^{\mE_{\varepsilon,R}}_{\mF_\varepsilon\oplus\mF_{\varepsilon,1}^\perp,\beta,\gamma}+\frac{\widehat c(\widetilde\sigma)}{\beta}
+\frac{W_{\widehat f_{\varepsilon,R}}}{\beta}\right)\left(\varphi_{\varepsilon,R,1}s\right)\right\|
\\
+\left\| \left(D ^{\mE_{\varepsilon,R}}_{\mF_\varepsilon\oplus\mF_{\varepsilon,1}^\perp,\beta,\gamma}+\frac{\widehat c(\widetilde\sigma)}{\beta}
+\frac{W_{\widehat f_{\varepsilon,R}}}{\beta}\right)\left(\varphi_{\varepsilon,R,2}s\right)\right\| 
-\left\|c_{\beta,\gamma}\left(d\varphi_{\varepsilon,R,1}\right)s\right\|
-\left\|c_{\beta,\gamma}\left(d\varphi_{\varepsilon,R,2}\right)s\right\|,
\end{multline}
where we identify a one form with its gradient.

Let $f_1, \dots, f_q$ (resp. $h_1,\dots,h_{q_1}$; resp. $e_1,\dots, e_{q_2}$) be an orthonormal basis of 
$(\mF_\varepsilon,g^{\mF_\varepsilon})$ (resp. $(\mF_{\varepsilon,1}^\perp, g^{\mF_{\varepsilon,1}^\perp})$; resp. $(\mF_{\varepsilon,2}^\perp, g^{\mF_{\varepsilon,2}^\perp})$). Then by \cite[(2.17)]{Z17} 
one has  
\begin{multline}\label{0.23}
\left(D ^{\mE_{\varepsilon,R}}_{\mF_\varepsilon\oplus\mF_{\varepsilon,1}^\perp,\beta,\gamma}+\frac{\widehat c(\widetilde\sigma)}{\beta}
+\frac{W_{\widehat f_{\varepsilon,R}}}{\beta}\right)^2=\left(D ^{\mE_{\varepsilon,R}}_{\mF_\varepsilon\oplus\mF_{\varepsilon,1}^\perp,\beta,\gamma}+\frac{\widehat c(\widetilde\sigma)}{\beta}\right)^2
\\
+\sum_{i=1}^q \beta^{-1}c_{\beta,\gamma}\left(\beta^{-1}f_i\right)\left[\nabla^{\mE_{\varepsilon,R}}_{f_i}, \frac{W_{\widehat f_{\varepsilon,R}}}{\beta}\right]+\sum_{s=1}^{q_1}\gamma\, c_{\beta,\gamma}(\gamma h_s)\left[\nabla^{\mE_{\varepsilon,R}}_{h_s},\frac{W_{\widehat f_{\varepsilon,R}}}{\beta}\right]
\\
+\sum_{j=1}^{q_2}c(e_j)\left[\nabla^{\mE_{\varepsilon,R}}_{e_j},\frac{W_{\widehat f_{\varepsilon,R}}}{\beta}\right]
+\frac{W^2_{\widehat f_{\varepsilon,R}}}{\beta^2}.
\end{multline}

From  (\ref{0.5}),  (\ref{0.9}) and (\ref{0.14}), one has
\begin{align}\label{0.24a}
 \left[\nabla^{\mE_{\varepsilon,R}}, W_{\widehat f_{\varepsilon,R}} \right]
=\widehat f_{\varepsilon,R}^*\left(\left[\nabla^{E_{0}} +\nabla^{E_{1}}  , W \right]\right)
=0\ \ 
{\rm on}\ \ \mM_{\mH_\varepsilon,R}',
\end{align}
while for any $X\in \mF^\perp_{\varepsilon,2} $, one has
\begin{align}\label{0.24b}
 \left[\nabla_X^{\mE_{\varepsilon,R}}, W_{\widehat f_{\varepsilon,R}} \right]=\widehat f_{\varepsilon,R}^*\left(\left[\nabla_{{{\widehat f_{\varepsilon,R\,*}}} (X)}^{E_{0}} +\nabla_{{{\widehat f_{\varepsilon,R\,*}}} (X)}^{E_{1}}  , W \right]\right)
=0\ \ 
{\rm on}\ \ \mM_{\mH_\varepsilon,R}.
\end{align}
Also, since $g^{\mF_\varepsilon}=\widetilde\pi_\varepsilon^*g^{F_\varepsilon}$, one has via (\ref{0.1}) and the first equality in (\ref{0.24b}) that
for any $X\in \mF_{\varepsilon} $,
\begin{align}\label{0.24c}
 \left[\nabla_X^{\mE_{\varepsilon,R}}, W_{\widehat f_{\varepsilon,R}} \right]=O(
\varepsilon|X|)\ \ 
{\rm on}\ \ \mM_{\mH_\varepsilon,R},
\end{align} 
and that for  any $X\in \mF^\perp_{\varepsilon,1} $,
\begin{align}\label{0.24d}
 \left[\nabla_X^{\mE_{\varepsilon,R}}, W_{\widehat f_{\varepsilon,R}} \right]=O_{\varepsilon,R}(
 |X|)\ \ 
{\rm on}\ \ \mM_{\mH_\varepsilon,R}.
\end{align}

From (\ref{0.24a})-(\ref{0.24d}), one gets\footnote{This is the formula which depends essentially on the enlargeability condition (\ref{0.1}), where  the so-called  area enlargeability is not enough.}
\begin{multline}\label{1.25}
 \sum_{i=1}^q \beta^{-1}c_{\beta,\gamma}\left(\beta^{-1}f_i\right)\left[\nabla^{\mE_{\varepsilon,R}}_{f_i}, \frac{W_{\widehat f_{\varepsilon,R}}}{\beta}\right] + \sum_{s=1}^{q_1}\gamma\, c_{\beta,\gamma}(\gamma h_s)\left[\nabla^{\mE_{\varepsilon,R}}_{h_s},\frac{W_{\widehat f_{\varepsilon,R}}}{\beta}\right] 
\\
+ \sum_{j=1}^{q_2}c(e_j)\left[\nabla^{\mE_{\varepsilon,R}}_{e_j},\frac{W_{\widehat f_{\varepsilon,R}}}{\beta}\right] 
=O \left(\frac{\varepsilon}{\beta^2}\right)+O_{\varepsilon,R}\left(\frac{\gamma }{\beta}\right).
\end{multline}

Similarly, by proceeding as in (\ref{1.25}) and  \cite[(1.21)]{Z19}, one has for $j=1,2$ that
\begin{align}\label{1.26}
\left|c_{\beta,\gamma}\left(d\varphi_{\varepsilon,R,j}\right)\right|=O\left(\frac{\varepsilon}{\beta}\right)+O_{\varepsilon,R}(\gamma ).
\end{align}

From (\ref{0.23}) and (\ref{1.25}), one has
\begin{multline}\label{1.26a}
\sum_{j=1}^2\left\|\left(D ^{\mE_{\varepsilon,R}}_{\mF_\varepsilon\oplus\mF_{\varepsilon,1}^\perp,\beta,\gamma}+\frac{\widehat c(\widetilde\sigma)}{\beta}
+\frac{W_{\widehat f_{\varepsilon,R}}}{\beta}\right)\left(\varphi_{\varepsilon,R,j}s\right)\right\|^2
\\
=\sum_{j=1}^2\left\|\left(D ^{\mE_{\varepsilon,R}}_{\mF_\varepsilon\oplus\mF_{\varepsilon,1}^\perp,\beta,\gamma}+\frac{\widehat c(\widetilde\sigma)}{\beta}\right)\left(\varphi_{\varepsilon,R,j}s\right)\right\|^2
+\frac{1}{\beta^2} \left\|  {W_{\widehat f_{\varepsilon,R}}}s\right\|^2
+
O \left(\frac{\varepsilon}{\beta^2}\right)\|s\|^2
+O_{\varepsilon,R}\left(\frac{\gamma }{\beta}\right)
\|s\|^2.
\end{multline}

By (\ref{0.12})-(\ref{0.12a}) and proceeding as in \cite[p. 1058-1059]{Z17}, one gets
\begin{multline}\label{1.26b}
 \left\|\left(D ^{\mE_{\varepsilon,R}}_{\mF_\varepsilon\oplus\mF_{\varepsilon,1}^\perp,\beta,\gamma}+\frac{\widehat c(\widetilde\sigma)}{\beta}
 \right)\left(\varphi_{\varepsilon,R,1}s\right)\right\|^2
\geq \frac{\delta}{4\beta^2}\|\varphi_{\varepsilon,R,1}s\|^2
\\
+O_\varepsilon\left(\frac{1}{\beta^2R} \right)
\|\varphi_{\varepsilon,R,1}s\|^2
+
O \left(\frac{\varepsilon^2}{\beta^2}\right)\|\varphi_{\varepsilon,R,1}s\|^2
+O_{\varepsilon,R}\left(\frac{1}{\beta}+\frac{\gamma^2}{\beta^2}\right)
\|\varphi_{\varepsilon,R,1}s\|^2.
\end{multline}

From (\ref{0.18}), we know that 
\begin{align}\label{1.26c}
 \left\| \varphi_{\varepsilon,R,2} {W_{\widehat f_{\varepsilon,R}}}s\right\|^2\geq \delta_1\|\varphi_{\varepsilon,R,2}s\|^2.
\end{align}

From (\ref{1.26a})-(\ref{1.26c}), one finds
\begin{multline}\label{1.26d}
\sum_{j=1}^2\left\|\left(D ^{\mE_{\varepsilon,R}}_{\mF_\varepsilon\oplus\mF_{\varepsilon,1}^\perp,\beta,\gamma}+\frac{\widehat c(\widetilde\sigma)}{\beta}
+\frac{W_{\widehat f_{\varepsilon,R}}}{\beta}\right)\left(\varphi_{\varepsilon,R,j}s\right)\right\|^2
\geq 
\frac{{\rm min}\{\frac{\delta}{4},\delta_1\}}{\beta^2} \left\|  s\right\|^2
\\
+
O \left(\frac{\varepsilon}{\beta^2}\right)\|s\|^2
+O_\varepsilon\left(\frac{1}{\beta^2R} \right)
\| s\|^2
+O_{\varepsilon,R}\left(\frac{1}{\beta}+\frac{\gamma^2}{\beta^2}\right)
\| s\|^2
.
\end{multline}

From (\ref{0.22}), (\ref{1.26}) and (\ref{1.26d}), one gets (\ref{0.16}) easily.

To prove (\ref{0.17}), 
  for any smooth section $s$ in question, one has as in   (\ref{0.22}) that
\begin{multline}\label{1.26e}
 \sqrt{2}\left\|\left( h\left(\frac{\rho}{R}\right)  D ^{\mE_{\varepsilon,R}}_{\mF_\varepsilon\oplus\mF_{\varepsilon,1}^\perp,\beta,\gamma} h\left(\frac{\rho}{R}\right)  
+\frac{\widehat c(\widetilde\sigma)}{\beta}
+\frac{W_{\widehat f_{\varepsilon,R}}}{\beta}\right)s\right\|
\\
\geq \left\|\left(h\left(\frac{\rho}{R}\right)  D ^{\mE_{\varepsilon,R}}_{\mF_\varepsilon\oplus\mF_{\varepsilon,1}^\perp,\beta,\gamma}h\left(\frac{\rho}{R}\right)  
+\frac{\widehat c(\widetilde\sigma)}{\beta}
+\frac{W_{\widehat f_{\varepsilon,R}}}{\beta}\right)\left(\varphi_{\varepsilon,R,1}s\right)\right\|
\\
+\left\| \left(h\left(\frac{\rho}{R}\right)  D ^{\mE_{\varepsilon,R}}_{\mF_\varepsilon\oplus\mF_{\varepsilon,1}^\perp,\beta,\gamma}h\left(\frac{\rho}{R}\right)  
+\frac{\widehat c(\widetilde\sigma)}{\beta}
+\frac{W_{\widehat f_{\varepsilon,R}}}{\beta}\right)\left(\varphi_{\varepsilon,R,2}s\right)\right\| 
-\left\|c_{\beta,\gamma}\left(d\varphi_{\varepsilon,R,1}\right)s\right\|
\\
-\left\|c_{\beta,\gamma}\left(d\varphi_{\varepsilon,R,2}\right)s\right\|.
\end{multline}

Clearly (cf. \cite[(2.29)]{Z17}),
\begin{multline}\label{0.25}
\left(h\left(\frac{\rho}{R}\right) D ^{\mE_{\varepsilon,R}}_{\mF_\varepsilon\oplus\mF_{\varepsilon,1}^\perp,\beta,\gamma}
 h\left(\frac{\rho}{R}\right) 
+\frac{\widehat c\left(\widetilde\sigma\right)}{\beta}
+\frac{W_{\widehat f_{\varepsilon,R}}}{\beta}\right)^2
\\
=\left(h\left(\frac{\rho}{R}\right) D ^{\mE_{\varepsilon,R}}_{\mF_\varepsilon\oplus\mF_{\varepsilon,1}^\perp,\beta,\gamma}
 h\left(\frac{\rho}{R}\right) 
+\frac{\widehat c\left(\widetilde\sigma\right)}{\beta}\right)^2
+h\left(\frac{\rho}{R}\right) ^2\left[D ^{\mE_{\varepsilon,R}}_{\mF_\varepsilon\oplus\mF_{\varepsilon,1}^\perp,\beta,\gamma}
  ,\frac{W_{\widehat f_{\varepsilon,R}}}{\beta}\right]
 +\frac{W^2_{\widehat f_{\varepsilon,R}}}{\beta^2}
 \\
 =\left(h\left(\frac{\rho}{R}\right) D ^{\mE_{\varepsilon,R}}_{\mF_\varepsilon\oplus\mF_{\varepsilon,1}^\perp,\beta,\gamma}
 h\left(\frac{\rho}{R}\right)\right)^2+\frac{h(\frac{\rho}{R})^2}{\beta}\left[D ^{\mE_{\varepsilon,R}}_{\mF_\varepsilon\oplus\mF_{\varepsilon,1}^\perp,\beta,\gamma},\widehat c\left(\widetilde\sigma\right)\right]+\frac{\left|\widetilde\sigma\right|^2}{\beta^2}
 \\
+\frac{h\left(\frac{\rho}{R}\right)^2}{\beta} \left[ D ^{\mE_{\varepsilon,R}}_{\mF_\varepsilon\oplus\mF_{\varepsilon,1}^\perp,\beta,\gamma}
  , {W_{\widehat f_{\varepsilon,R}}}{ }\right]
 +\frac{W^2_{\widehat f_{\varepsilon,R}}}{\beta^2}.
  \end{multline}

From (\ref{1.25}) and the first equality in (\ref{0.25}), one has
\begin{multline}\label{0.25a}
\left\|\left(h\left(\frac{\rho}{R}\right) D ^{\mE_{\varepsilon,R}}_{\mF_\varepsilon\oplus\mF_{\varepsilon,1}^\perp,\beta,\gamma}
 h\left(\frac{\rho}{R}\right) 
+\frac{\widehat c\left(\widetilde\sigma\right)}{\beta}
+\frac{W_{\widehat f_{\varepsilon,R}}}{\beta}\right)\left(\varphi_{\varepsilon,R,2}s\right)\right\|^2
\\
\geq 
\frac{1}{\beta^2} \left\|  \varphi_{\varepsilon,R,2}W_{\widehat f_{\varepsilon,R}}s \right\|^2
 +O \left(\frac{\varepsilon}{\beta^2}\right)\|s\|^2+O_{\varepsilon,R}\left(\frac{\gamma }{\beta}\right)\|s\|^2.
  \end{multline}

By proceeding as in \cite[(2.27)]{Z17}, one has on 
$  \mM_{\mH_\varepsilon,R}\setminus s(M_{H_\varepsilon})$ that
\begin{align}\label{0.25b}
  \left[D ^{\mE_{\varepsilon,R}}_{\mF_\varepsilon\oplus\mF_{\varepsilon,1}^\perp,\beta,\gamma},\widehat c\left( \sigma\right)\right]
=O_\varepsilon\left(\frac{1}{\beta R}\right) + O_{\varepsilon,R}(1).
\end{align}

From (\ref{0.13}),  (\ref{1.25}), the second equality in (\ref{0.25}) and (\ref{0.25b}), one gets
\begin{multline}\label{0.25c}
\left\|\left(h\left(\frac{\rho}{R}\right) D ^{\mE_{\varepsilon,R}}_{\mF_\varepsilon\oplus\mF_{\varepsilon,1}^\perp,\beta,\gamma}
 h\left(\frac{\rho}{R}\right) 
+\frac{\widehat c\left(\widetilde\sigma\right)}{\beta}
+\frac{W_{\widehat f_{\varepsilon,R}}}{\beta}\right)\left(\varphi_{\varepsilon,R,1}s\right)\right\|^2
\geq 
\frac{1}{\beta^2} \left\|  \varphi_{\varepsilon,R,1} s \right\|^2
\\
+
\frac{1}{\beta^2} \left\|  \varphi_{\varepsilon,R,1}W_{\widehat f_{\varepsilon,R}}s \right\|^2
 +O \left(\frac{\varepsilon}{\beta^2}\right)\|s\|^2
+O_\varepsilon\left(\frac{1}{\beta^2 R}\right) \|s\|^2+ O_{\varepsilon,R}\left(\frac{1}{\beta}\right)\|s\|^2.
  \end{multline}

From (\ref{1.26c}),  (\ref{0.25a}) and (\ref{0.25c}), one gets
\begin{multline}\label{0.25d}
\sum_{j=1}^2
\left\|\left(h\left(\frac{\rho}{R}\right) D ^{\mE_{\varepsilon,R}}_{\mF_\varepsilon\oplus\mF_{\varepsilon,1}^\perp,\beta,\gamma}
 h\left(\frac{\rho}{R}\right) 
+\frac{\widehat c\left(\widetilde\sigma\right)}{\beta}
+\frac{W_{\widehat f_{\varepsilon,R}}}{\beta}\right)\left(\varphi_{\varepsilon,R,j}s\right)\right\|^2
\\
\geq 
\frac{1}{\beta^2} \left\|  \varphi_{\varepsilon,R,1} s \right\|^2
+
\frac{1}{\beta^2} \left\|   W_{\widehat f_{\varepsilon,R}}s \right\|^2
 +O \left(\frac{\varepsilon}{\beta^2}\right)\|s\|^2
+O_\varepsilon\left(\frac{1}{\beta^2 R}\right) \|s\|^2+ O_{\varepsilon,R}\left(\frac{1}{\beta}\right)\|s\|^2
\\
\geq
\frac{{\rm min}\{1,\delta_1\}}{\beta^2} \left\|   s \right\|^2
 +O \left(\frac{\varepsilon}{\beta^2}\right)\|s\|^2
+O_\varepsilon\left(\frac{1}{\beta^2 R}\right) \|s\|^2+ O_{\varepsilon,R}\left(\frac{1}{\beta}\right)\|s\|^2.
  \end{multline}

From (\ref{1.26}) and (\ref{0.25d}), one gets (\ref{0.17}) easily. 
\end{proof}

\subsection{Elliptic operators on
$\widetilde \mN_{ \varepsilon,R}$}\label{s1.4}
Let $Q$ be a Hermitian vector bundle over $\widehat \mM_{\mH_\varepsilon,R}$ such that $\left(S_{\beta,\gamma}  \left(\mF_\varepsilon\oplus\mF_{\varepsilon,1}^\perp\right)\widehat\otimes
\Lambda^*\left (\mF_{\varepsilon,2}^\perp \right)\widehat \otimes \mE_{\varepsilon,R}\right)_{-}\oplus Q$ is a trivial vector bundle over $\widehat \mM_{\mH_\varepsilon,R}$. Then $\left(S_{\beta,\gamma}  \left(\mF_\varepsilon\oplus\mF_{\varepsilon,1}^\perp\right)\widehat\otimes
\Lambda^*\left (\mF_{\varepsilon,2}^\perp \right)\widehat \otimes \mE_{\varepsilon,R}\right)_{+}\oplus Q$ is a trivial vector bundle near $\partial \widehat \mM_{\mH_\varepsilon,R}$, under the identification $\widehat c(\sigma)+\widehat f_{\varepsilon,R}^*(w)+{\rm Id}_Q$.

By obviously extending the above trivial vector bundles to $\mN _{\varepsilon,R}$, we get a ${\bf Z}_2$-graded Hermitian vector bundle $\xi=\xi_+\oplus \xi_-$ over $\widetilde  \mN_{\varepsilon,R}$ and an odd self-adjoint endomorphism $V=v+v^*\in \Gamma({\rm End}(\xi))$ (with $v:\Gamma(\xi_+)\to \Gamma(\xi_-)$, $v^*$ being the adjoint of $v$) such that 
\begin{align}\label{0.26}
\xi_\pm=\left(S_{\beta,\gamma}  \left(\mF_\varepsilon\oplus\mF_{\varepsilon,1}^\perp\right)\widehat\otimes
\Lambda^*\left (\mF_{\varepsilon,2}^\perp \right)\widehat \otimes \mE_{\varepsilon,R}\right)_\pm\oplus Q
\end{align}
over $\widehat \mM_{\mH_\varepsilon,R}$, $V$ is invertible on $\mN _{\varepsilon,R}$ and 
\begin{align}\label{0.27}
V=  \widehat c\left(\widetilde\sigma\right)+W_{\widehat f_{\varepsilon,R}}+\left(\begin{array}{clcr} 0 & {\rm Id}_Q\\
 {\rm Id}_Q & 0\\ \end{array} \right)
\end{align}
on $\widehat \mM_{\mH_\varepsilon,R}$, which is invertible on $\widehat\mM_{\mH_\varepsilon,R}\setminus  \mM_{\mH_\varepsilon,\frac{R}{2}}$.

Recall that $h(\frac{\rho}{R})$ vanishes near $\mM_{\mH_\varepsilon,R}\cap \partial  \mM_{\varepsilon,R}$. We extend it to a function on $\widetilde \mN_{ \varepsilon,R}$ which equals to zero on $\mN _{\varepsilon,R}$ and an open neighborhood of $\partial \widehat \mM_{\mH_\varepsilon,R}$ in $\widetilde \mN_{ \varepsilon,R}$, and we denote the resulting function on $\widetilde \mN_{ \varepsilon,R}$ by $\widetilde h_R$.

 Let $\pi_{\widetilde \mN_{ \varepsilon,R}}: T\widetilde \mN_{ \varepsilon,R}\to \widetilde \mN_{ \varepsilon,R}$ be the projection of the tangent bundle of $\widetilde \mN_{ \varepsilon,R}$. Let $\gamma^{\widetilde \mN_{ \varepsilon,R}}\in {\rm Hom} (\pi^*_{\widetilde \mN_{ \varepsilon,R}}\xi_+,\pi^*_{\widetilde \mN_{ \varepsilon,R}}\xi_-)$ be the symbol defined by 
\begin{align}\label{0.28}
\gamma^{\widetilde \mN_{ \varepsilon,R}}(p,u)=\pi^*_{\widetilde \mN_{ \varepsilon,R}}\left(\sqrt{-1} \widetilde h ^2_R c_{\beta,\gamma}(u)+v(p)\right)\ \ {\rm for}\ \ p\in \widetilde \mN_{ \varepsilon,R},\ \ u\in T_p \widetilde \mN_{ \varepsilon,R}.
\end{align}
By (\ref{0.27}) and (\ref{0.28}), $\gamma^{\widetilde\mN_{ \varepsilon,R}}$ is singular only if $u=0$ and $p\in\mM_{\mH_\varepsilon,\frac{R}{2}}$. Thus $\gamma^{\widetilde \mN_{ \varepsilon,R}}$ is an elliptic symbol.

On the other hand, it is clear that $\widetilde h_R  D ^{\mE_{\varepsilon,R}}_{\mF_\varepsilon\oplus\mF_{\varepsilon,1}^\perp,\beta,\gamma}\widetilde h_R$ is well defined on $\widetilde \mN_{ \varepsilon,R}$ if we define it to equal to zero on $\widetilde \mN_{ \varepsilon,R}\setminus \widehat \mM_{\mH_\varepsilon,R}$.

Let $A: L^2 (\xi)\to L^2 (\xi)$ be a second order positive elliptic differential operator on $\widetilde \mN_{ \varepsilon,R}$ preserving the ${\bf Z}_2$-grading of $\xi=\xi_+\oplus \xi_-$, such that its symbol equals to $|\eta|^2$ at $\eta\in T\widetilde \mN_{ \varepsilon,R}$.\footnote{To be more precise, here $A$ also depends on the defining metric. We omit the corresponding subscript/superscript only for convenience.}
As in \cite[(2.33)]{Z17}, let $P^{\mE_{\varepsilon,R}}_{\varepsilon, R,\beta,\gamma}:L^2 (\xi)\to L^2 (\xi)$ be the zeroth order pseudodifferential operator on $\widetilde \mN_{ \varepsilon,R}$ defined by
\begin{align}\label{0.29}
P^{\mE_{\varepsilon,R}}_{\varepsilon,R,\beta,\gamma}=A^{-\frac{1}{4}}\widetilde h_R D ^{\mE_{\varepsilon,R}}_{\mF_\varepsilon\oplus\mF_{\varepsilon,1}^\perp,\beta,\gamma}\widetilde h_R A^{-\frac{1}{4}}+\frac{V}{\beta}.
\end{align}
Let $P^{\mE_{\varepsilon,R}}_{\varepsilon,R,\beta,\gamma,+}:L^2(\xi_+)\to L^2(\xi_-)$ be the obvious restriction. Then the principal symbol of $P^{\mE_{\varepsilon,R}}_{\varepsilon,R,\beta,\gamma,+}$, which we denote by $\gamma(P^{\mE_{\varepsilon,R}}_{\varepsilon,R,\beta,\gamma,+})$, is homotopic through elliptic symbols to $\gamma^{\widetilde \mN_{ \varepsilon,R}}$. Thus $P^{\mE_{\varepsilon,R}}_{\varepsilon,R,\beta,\gamma,+}$ is a Fredholm operator. Moreover, by the Atiyah-Singer index theorem \cite{ASI} (cf. \cite[Th. 13.8 of Ch. III]{LaMi89}) and the computation in \cite[\S 5]{GL83}, one finds
\begin{multline}\label{0.30}
{\rm ind}\left(P^{\mE_{\varepsilon,R}}_{\varepsilon,R,\beta,\gamma,+}\right)
={\rm ind}\left(\gamma\left(P^{\mE_{\varepsilon,R}}_{\varepsilon,R,\beta,\gamma,+}\right)\right)
={\rm ind}\left(\gamma^{\widetilde \mN_{ \varepsilon,R}}\right)
\\
=\left\langle \widehat A\left(T\widehat M_{H_\varepsilon}\right) 
\left({\rm ch}\left( f_\varepsilon^*E_0 \right)-{\rm ch}\left( f_\varepsilon^*E_1  \right)\right),\left[\widehat M_{H_\varepsilon}\right]\right\rangle
=\left({\rm deg}(f_\varepsilon)\right)\left\langle {\rm ch}\left(E_0\right),\left[S^{\dim M}(1)\right]\right\rangle
\neq 0,
\end{multline}
where the inequality comes from (\ref{0.4}).

For any $0\leq t\leq 1$, set
\begin{align}\label{0.31}
P^{\mE_{\varepsilon,R}}_{\varepsilon,R,\beta,\gamma,+}(t)=P^{\mE_{\varepsilon,R}}_{\varepsilon,R,\beta,\gamma,+}+\frac{(t-1)v}{\beta}+A^{-\frac{1}{4}}\frac{(1-t)v}{\beta}A^{-\frac{1}{4}}.
\end{align}
Then $P^{\mE_{\varepsilon,R}}_{\varepsilon,R,\beta,\gamma,+}(t)$ is a smooth family of zeroth order pseudodifferential operators such that the corresponding symbol $\gamma(P^{\mE_{\varepsilon,R}}_{\varepsilon,R,\beta,\gamma,+}(t))$ is elliptic for $0<t\leq 1$. Thus $P^{\mE_{\varepsilon,R}}_{\varepsilon,R,\beta,\gamma,+}(t)$ is a continuous family of Fredholm operators for $0<t\leq 1$ with $P^{\mE_{\varepsilon,R}}_{\varepsilon,R,\beta,\gamma,+}(1)=P^{\mE_{\varepsilon,R}}_{\varepsilon,R,\beta,\gamma,+}$.

Now since $P^{\mE_{\varepsilon,R}}_{\varepsilon,R,\beta,\gamma,+}(t)$ is continuous on the whole $[0,1]$, if $P^{\mE_{\varepsilon,R}}_{\varepsilon,R,\beta,\gamma,+}(0)$ is Fredholm and has vanishing index, then we would reach a contradiction with respect to (\ref{0.30}), and then complete the proof of Theorem \ref{t0.1}.

Thus we need only to prove the following analogue of \cite[Proposition 2.5]{Z17}.

\begin{prop}\label{t0.5}
There exist $\varepsilon, R,\beta,\gamma>0$ such that the following identity holds:
\begin{align}\label{0.32}
\dim\left({\rm ker}\left(P^{\mE_{\varepsilon,R}}_{\varepsilon,R,\beta,\gamma,+}(0)\right)\right)=\dim\left({\rm ker}\left(P^{\mE_{\varepsilon,R}}_{\varepsilon,R,\beta,\gamma,+}(0)^*\right)\right)=0.
\end{align}
\end{prop}

\begin{proof}
Let $P^{\mE_{\varepsilon,R}}_{\varepsilon,R,\beta,\gamma}(0): L^2 (\xi)\to L^2 (\xi)$ be given by
\begin{align}\label{0.33}
P^{\mE_{\varepsilon,R}}_{\varepsilon,R,\beta,\gamma}(0)=A^{-\frac{1}{4}}\widetilde h_R D ^{\mE_{\varepsilon,R}}_{\mF_\varepsilon\oplus\mF_{\varepsilon,1}^\perp,\beta,\gamma}\widetilde h_R A^{-\frac{1}{4}}+A^{-\frac{1}{4}}\frac{V}{\beta}A^{-\frac{1}{4}}.
\end{align}

Since $P^{\mE_{\varepsilon,R}}_{\varepsilon,R,\beta,\gamma}(0)$ is formally self-adjoint, by (\ref{0.29}) and (\ref{0.31}) we need only to show that 
\begin{align}\label{0.34}
\dim\left({\rm ker}\left(P^{\mE_{\varepsilon,R}}_{\varepsilon,R,\beta,\gamma}(0)\right)\right)=0
\end{align}
for certain $\varepsilon, R, \beta,\gamma>0$. 

Let $s\in {\rm ker}(P^{\mE_{\varepsilon,R}}_{\varepsilon,R,\beta,\gamma}(0))$. By (\ref{0.33}) one has
\begin{align}\label{0.35}
\left(\widetilde h_R D ^{\mE_{\varepsilon,R}}_{\mF_\varepsilon\oplus\mF_{\varepsilon,1}^\perp,\beta,\gamma}\widetilde h_R+\frac{V}{\beta}\right)A^{-\frac{1}{4}}s=0.
\end{align}

Since $\widetilde h_R=0$ on $\widetilde \mN_{ \varepsilon,R}\setminus \widehat \mM_{\mH_\varepsilon,R}$, while $V$ is invertible on $\widetilde \mN_{ \varepsilon,R}\setminus \widehat \mM_{\mH_\varepsilon,R}$, by (\ref{0.35}) one has
\begin{align}\label{0.36}
A^{-\frac{1}{4}}s=0\ \ {\rm on}\ \ \widetilde \mN_{ \varepsilon,R}\setminus \widehat \mM_{\mH_\varepsilon,R}.
\end{align}

Write on $\widehat \mM_{\mH_\varepsilon,R}$ that
\begin{align}\label{0.37}
A^{-\frac{1}{4}}s=s_1+s_2,
\end{align}
with $s_1\in L^2 (S_{\beta,\gamma}  \left(\mF_\varepsilon\oplus\mF_{\varepsilon,1}^\perp\right)\widehat\otimes
\Lambda^*\left (\mF_{\varepsilon,2}^\perp \right)\widehat \otimes \mE_{\varepsilon,R})$ and $s_2\in L^2 (Q\oplus Q)$.

By (\ref{0.27}), (\ref{0.35}) and (\ref{0.37}), one has
\begin{align}\label{0.38}
s_2=0,
\end{align}
while
\begin{align}\label{0.39}
\left(\widetilde h_R D ^{\mE_{\varepsilon,R}}_{\mF_\varepsilon\oplus\mF_{\varepsilon,1}^\perp,\beta,\gamma}\widetilde h_R+\frac{ \widehat c\left(\widetilde\sigma\right)}{\beta}+\frac{W_{\widehat f_{\varepsilon,R}}}{\beta}\right)s_1=0.
\end{align}

We need to show that (\ref{0.39}) implies $s_1=0$.  

As in (\ref{1.26e}), one has 
\begin{multline}\label{0.40}
 \sqrt{2}\left\|\left(\widetilde h_R  D ^{\mE_{\varepsilon,R}}_{\mF_\varepsilon\oplus\mF_{\varepsilon,1}^\perp,\beta,\gamma} \widetilde h_R 
+\frac{\widehat c(\widetilde\sigma)}{\beta}
+\frac{W_{\widehat f_{\varepsilon,R}}}{\beta}\right)s_1\right\|
\\
\geq \left\|\left(\widetilde h_R  D ^{\mE_{\varepsilon,R}}_{\mF_\varepsilon\oplus\mF_{\varepsilon,1}^\perp,\beta,\gamma}\widetilde h_R 
+\frac{\widehat c(\widetilde\sigma)}{\beta}
+\frac{W_{\widehat f_{\varepsilon,R}}}{\beta}\right)\left(\varphi_{\varepsilon,R,1}s_1\right)\right\|
\\
+\left\| \left(\widetilde h_R   D ^{\mE_{\varepsilon,R}}_{\mF_\varepsilon\oplus\mF_{\varepsilon,1}^\perp,\beta,\gamma}\widetilde h_R  
+\frac{\widehat c(\widetilde\sigma)}{\beta}
+\frac{W_{\widehat f_{\varepsilon,R}}}{\beta}\right)\left(\varphi_{\varepsilon,R,2}s_1\right)\right\| 
-\left\|c_{\beta,\gamma}\left(d\varphi_{\varepsilon,R,1}\right)s_1\right\|
\\
-\left\|c_{\beta,\gamma}\left(d\varphi_{\varepsilon,R,2}\right)s_1\right\|.
\end{multline}

By proceeding as in the proof of (\ref{0.25a}), one gets
\begin{multline}\label{0.41}
\left\|\left(\widetilde h_R D ^{\mE_{\varepsilon,R}}_{\mF_\varepsilon\oplus\mF_{\varepsilon,1}^\perp,\beta,\gamma}
 \widetilde h_R 
+\frac{\widehat c\left(\widetilde\sigma\right)}{\beta}
+\frac{W_{\widehat f_{\varepsilon,R}}}{\beta}\right)\left(\varphi_{\varepsilon,R,2}s_1\right)\right\|^2
\\
\geq 
\frac{1}{\beta^2} \left\|  \varphi_{\varepsilon,R,2}W_{\widehat f_{\varepsilon,R}}s_1 \right\|^2
 +O \left(\frac{\varepsilon}{\beta^2}\right)\|s_1\|^2+O_{\varepsilon,R}\left(\frac{\gamma }{\beta}\right)\|s_1\|^2.
  \end{multline}

On the other hand, by using Lemma \ref{t0.4} and proceeding as in \cite[p. 1062]{Z17}, one finds that there exist $c_1>0$, $\varepsilon>0$ and $R>0$ such that when 
$\beta,\,\gamma>0$ are sufficiently small, one has 
\begin{align}\label{0.42}
\left\|\left(\widetilde h_R D ^{\mE_{\varepsilon,R}}_{\mF_\varepsilon\oplus\mF_{\varepsilon,1}^\perp,\beta,\gamma}
 \widetilde h_R 
+\frac{\widehat c\left(\widetilde\sigma\right)}{\beta}
+\frac{W_{\widehat f_{\varepsilon,R}}}{\beta}\right)\left(\varphi_{\varepsilon,R,1}s_1\right)\right\| 
\geq 
\frac{c_1}{\beta} \left\|  \varphi_{\varepsilon,R,1} s_1 \right\| .
  \end{align}

From (\ref{1.26}), (\ref{1.26c}) and (\ref{0.40})-(\ref{0.42}), 
one finds that there exist $c_2>0$, $\varepsilon>0$ and $R>0$ such that when 
$\beta,\,\gamma>0$ are sufficiently small, one has 
\begin{align}\label{0.43}
\left\|\left(\widetilde h_R D ^{\mE_{\varepsilon,R}}_{\mF_\varepsilon\oplus\mF_{\varepsilon,1}^\perp,\beta,\gamma}
 \widetilde h_R 
+\frac{\widehat c\left(\widetilde\sigma\right)}{\beta}
+\frac{W_{\widehat f_{\varepsilon,R}}}{\beta}\right)  s_1 \right\| 
\geq 
\frac{c_2}{\beta} \left\|  s_1 \right\| ,
  \end{align}
which implies, via (\ref{0.39}), $s_1=0$. 
\end{proof}

\begin{rem}\label{t0.6}
 By combining the above method with what in \cite[\S 2.5]{Z17}, one gets a proof of Theorem \ref{t0.3}. We leave the details to the interested reader.   
\end{rem} 

\begin{rem}\label{t0.7}
From the above proof, one sees that for Theorems \ref{t0.1} and \ref{t0.3} to hold, one need only to assume that (\ref{0.1}) holds for $X\in \Gamma(F_\varepsilon)$.  Moreover, when $M$ is compact and $M_\varepsilon$ might be noncompact, the above proof can also be seen as to complete in details the proof of the main results in \cite{Z19} for non-compactly enlargeable foliations. 
\end{rem} 

$\ $

\noindent{\bf Acknowledgments.}  The authors would like to thank the anonymous referee for careful reading and valuable suggestions. This work was partially supported by NSFC Grant No. 12271266, NSFC Grant No. 11931007 and Nankai Zhide Foundation.

\def\cprime{$'$} \def\cprime{$'$}
\providecommand{\bysame}{\leavevmode\hbox to3em{\hrulefill}\thinspace}
\providecommand{\MR}{\relax\ifhmode\unskip\space\fi MR }
\providecommand{\MRhref}[2]{%
  \href{http://www.ams.org/mathscinet-getitem?mr=#1}{#2}
}
\providecommand{\href}[2]{#2}

\end{document}